\documentclass[a4paper,10pt]{article}
\usepackage[utf8]{inputenc}
\usepackage{amsmath,amsfonts,amssymb,amscd,amsthm,xspace, graphicx}
\usepackage[normalem]{ulem}
\usepackage[mathscr]{eucal}
\usepackage{color}
\usepackage{ulem}
\usepackage[table]{xcolor}

 \usepackage[pagebackref,
            colorlinks,linkcolor={blue},citecolor={red}]
           {hyperref}
\newtheorem{theorem}{Theorem}
\newtheorem{lemma}[theorem]{Lemma}
\newtheorem{corollary}[theorem]{Corollary}

\theoremstyle{remark}

\numberwithin{equation}{section}

\def\HI{{\mathscr{F}}}

\def\0{{\mathbf{0}}}

\def\R{{\mathbb R}}

\def\et{{\tilde{\mathbf e}}}

\def\FF{\mathbf{F}_L}

\def\div{\hbox{\rm div}\,}

\def\ve{\mathbf v}
\def\vp{\varphi}

\def\ee{\mathbf e}

\def\e{\varepsilon}

\def\ve{{\mathbf v}}
\def\e{\epsilon}
\def\we{{\mathbf w}}
\def\wl{{\mathbf w}_L}
\def\Pl{\Phi_L}
\def\pl{p_L}
\def\ww{{\mathbf w}_0}
\def\0{{\mathbf 0}}
\def\wi{{\mathbf w}_\infty}

\def\div{\hbox{\rm div}\,}

\newcommand{\meas}{\mathop{\mathrm{meas}}}

\def\Xint#1{\mathchoice
{\XXint\displaystyle\textstyle{#1}}%
{\XXint\textstyle\scriptstyle{#1}}%
{\XXint\scriptstyle\scriptscriptstyle{#1}}%
{\XXint\scriptscriptstyle\scriptscriptstyle{#1}}%
\!\int}
\def\XXint#1#2#3{{\setbox0=\hbox{$#1{#2#3}{\int}$ }
\vcenter{\hbox{$#2#3$ }}\kern-.6\wd0}}

\def\dashint{\Xint-}

\begin{document}

\title{Leray's plane stationary solutions at small Reynolds numbers}
\author{Mikhail Korobkov\footnotemark[1] \and Xiao Ren\footnotemark[2]}
\renewcommand{\thefootnote}{\fnsymbol{footnote}}
\footnotetext[1]{
School of Mathematical Sciences,
Fudan University, Shanghai 200433, P. R.China; and Sobolev Institute of Mathematics, pr-t Ac. Koptyug, 4, Novosibirsk, 630090, Russia.  email: korob@math.nsc.ru}
\footnotetext[2]{School of Mathematical Sciences,
Fudan University, Shanghai 200433, P. R.China. email: xiaoren18@fudan.edu.cn}

\maketitle

\begin{abstract}
 In the celebrated 1933 paper \cite{Leray} by J. Leray, the invading domains method was proposed to construct $D$-solutions for the stationary Navier-Stokes flow around obstacle problem. In two dimensions, whether Leray's $D$-solution achieves the prescribed limiting velocity at spatial infinity became a major open problem since then. In this paper, we solve this problem at small Reynolds numbers. The proof builds on a novel blow-down argument which rescales the invading domains to the unit disc, as well as the ideas developed in the~recent works \cite{KPR20, KR21}. 
\end{abstract}
\renewcommand{\thefootnote}{\arabic{footnote}} 
\section*{Introduction}
This paper studies the~stationary Navier-Stokes equations in an exterior domain $\Omega$, with Dirichlet boundary condition on $\partial \Omega$ and nonzero prescribed velocity at spatial infinity, that is, 
\begin{equation}  \label{NSE}
\left\{
\begin{aligned}
 & \Delta \mathbf{w} - (\mathbf{w} \cdot \nabla) \mathbf{w} - \nabla p = 0, \\
 & \nabla \cdot \mathbf{w} = 0, \\
 & \mathbf{w}|_{\partial \Omega} = 0, \\
 & \mathbf{w}(z) \to\mathbf{w}_\infty=\lambda \mathbf{e}_1 \ \ \text{as}\ \  |z| \to \infty.
\end{aligned}
\right.
\end{equation} 
The parameter $\lambda > 0$ will be referred to as the Reynolds number. Here $\mathbf{e}_1=(1,0)$ is the unit vector along $x$-axis, and the~considered domain is of type $\Omega = \mathbb{R}^2 \setminus \bar{U}$, where $U$ is an open bounded domain with sufficiently smooth ($C^{2+\alpha}, \alpha > 0$ would be sufficient) boundary. Physically, the system \eqref{NSE} describes the stationary motion of a viscous incompressible fluid flowing past a~rigid cylindrical body, a classical problem in fluid mechanics. In the prominent lecture by professor V.I.~Yudovich, which he gave at the University of Cambridge and published in~\cite{Y}, the existence of solutions to~(\ref{NSE}) for arbitrary $\lambda>0$ was included in the list of the most important open problems in mathematical fluid mechanics.

Dating back to 1933, J. Leray \cite{Leray} suggested the following elegant approach to construct solutions to \eqref{NSE}, called  ``{\it the~invading domains method}\,''. (This method is particularly natural from the point of view of physical applications and numerical simulations.) Consider the finite domains $\Omega_k = \Omega \cap B_{R_k}$ where $B_{R_k}$ is the disc centered at 0 with  radius $R_k\to+\infty$. These are called invading domains, since as $k \to \infty$, $\Omega_k$ approaches the infinite exterior domain $\Omega$. In \cite{Leray} Leray was able to establish the existence of solutions to the Navier-Stokes boundary value problem on $\Omega_k$,
\begin{equation}
\label{NSE_k}
\left\{\begin{array}{r@{}l}
- \Delta{{\mathbf w}_k}+({\mathbf w}_k\cdot\nabla){\mathbf w}_k+ \nabla p_k  & {} ={\bf 0}\qquad \hbox{\rm in } \Omega\cap B_{R_k}, \\[2pt]
\div{{\mathbf w}_k} & {} =0\,\qquad \hbox{\rm in } \Omega\cap B_{R_k},\\[2pt]
{{\mathbf w}_k} & {} = \mathbf0\,\qquad \hbox{\rm on } \partial\Omega,  \\[2pt]
{\mathbf w}_k & {}=\we_\infty\;\quad\mbox{for \ }|z|=R_k.
 \end{array}\right.
\end{equation}
Moreover, he proved that the sequence $\mathbf{w}_k$ has uniformly bounded Dirichlet energy, i.e.,
\begin{equation}
\int_{\Omega_k} |\nabla \mathbf{w}_k|^2  \le C < \infty
\end{equation}
for some constant $C$ depending only on $\partial \Omega$ and $\lambda$. As a consequence, he observes that it is possible to extract a subsequence of $\mathbf{w}_{k}$ which converges weakly\footnote{This convergence is uniform on every bounded set.} to some function $\mathbf{w}_L$ that solves \eqref{NSE}$_{1,2,3}$. The function~$\we_L$ is now referred to as a Leray solution. However, the following problem turns out to be challenging:
\smallskip

\emph{ Does $\mathbf{w}_L$ satisfy the limiting velocity condition \eqref{NSE}$_{4}$, i.e., do we have
\begin{equation} \label{eq:correctlimit}
\mathbf{w}_L(z) \to \mathbf{w}_\infty\ \ \mbox{\rm as}\ \ |z| \to \infty ?
\end{equation}}

Many useful properties of Leray solutions were discovered in the classical papers by D.~Gilbarg and H.F.~Weinberger \cite{GW1}--\cite{GW}. Further, in the very deep paper~\cite{Amick} Ch.Amick proved, under additional axial symmetry assumption, that the Leray solutions are nontrivial (i.e., they are not identically zero) and that they have some uniform limits at infinity, i.e., there exists a constant vector~$\mathbf{w}_0\in\R^2$ such that 
\begin{equation}  \label{NSE-cc1}  \mathbf{w}_L(z) \to \mathbf{w}_0 \ \ \text{as}\ \  |z| \to \infty.
\end{equation} 
Very recently, in the joint papers by Korobkov--Pileckas--Russo~\cite{KPR}--\cite{KPR20},  Amick's additional symmetry assumption in the above results was removed. 

Nevertheless, despite the classical papers and  the recent progress, the fundamental question, whether or not the Leray solutions satisfy the limiting condition~(\ref{eq:correctlimit}), i.e., whether the~equality 
$\mathbf{w}_0=\we_\infty$ holds, was still unanswered. In other words, it remained unclear  whether one can construct  genuine solutions to the initial problem~(\ref{NSE}) by Leray's method. 

In this paper, we establish the convergence \eqref{eq:correctlimit} and justify Leray's approach in the small Reynolds numbers case. Moreover, due to the recent uniqueness result \cite{KR21}, we can actually prove the convergence of $\mathbf{w}_k$ to the desired solution of \eqref{NSE} without the need of taking subsequences.

\begin{theorem} \label{thm_main}
{\sl There exists a constant $\lambda_1>0$ depending only on the geometry of $\partial \Omega$ such that, for any $0 < \lambda \le \lambda_1$, the invading domains solutions $(\mathbf{w}_k, p_k)$ to \eqref{NSE_k} converge weakly to the unique D-solution to \eqref{NSE}, as $k \to \infty$.}
\end{theorem}

The solutions to \eqref{NSE} for small Renolds numbers $\lambda$ were first constructed in the classical work \cite{FS} of Finn-Smith  in 1967. They used another approach, which is very different from Leray's invading domains method. This approach is based on delicate Oseen system estimates in the exterior domain and contraction mapping arguments. The relation between the solutions generated by the two distinct approaches (Finn-Smith versus Leray) was largely unclear since then.  Recently, in \cite{KR21} we proved  that the Finn-Smith solutions are the only $D$-solutions to \eqref{NSE} when $\lambda$ is sufficiently small, by developing the ideas in Amick's elegant work \cite{Amick}. It is interesting to point out that the original motivation of \cite{Amick} was to investigate Leray's approach.  Our Theorem \ref{thm_main} unifies these two approaches when $\lambda$ is small, by showing that they produce exactly the same solutions.

More detailed survey of results concerning boundary value problems for stationary NS-system in plane exterior domains can be found, $e.g.$, in \cite{G}, \cite{JG}. The subject is still a source of   interest, as evidenced, $e.g.$,  by the   papers  \cite{G2}, \cite{GuilW1}--\cite{HW}.

Now let us describe the main ideas and approaches of  the paper. 
We start with an interesting and very useful result stating that any~ $D$-solutions to \eqref{NSE_k} have \textit{extra} small Dirichlet energy when $\lambda$ is small:
\begin{equation} \label{Dlambda--iii}
   D_k= \int_{\Omega_k} |\nabla \mathbf{w}_k|^2 \le \frac{C\lambda^2}{|\log\lambda|}
 \end{equation}
 for $k$ large enough, where $C$ does not depend on $k$ or $\lambda$. 
 Note, that similar logarithmic smallness was proved in \cite[Theorem 22]{Amick}  for  the Dirichlet energy of the solutions to Stokes system, and in~\cite{KR21} for solutions to~(\ref{NSE}) in exterior domain.  

Further, we have to develop the methods of the recent paper~\cite{KPR20}, where it was proved that Leray solutions are always nontrivial. In particular, in~\cite{KPR20} it was proved that for a~fixed $\lambda>0$ the values of  Dirichlet integrals $D_k$ cannot be ``too small'', i.e., the estimate from below~$D_k\ge \sigma>0$ holds for some $\sigma$ independent of~$k$. In the present paper, by developing these methods, we prove  that the gap $\ww\ne\wi$ is impossible under smallness conditions~(\ref{Dlambda--iii}).

The main novelty here is a new estimate for limiting solutions to Euler equations, which is also of independent interest. Namely, consider the rescaled velocity field $\tilde{\mathbf{w}}_k(z) = \lambda^{-1} \mathbf{w}_k(R_k z)$, defined in the rescaled domains $\widetilde\Omega_k=\frac1{R_k}\Omega_k=B_1 \setminus (R_k^{-1} \Omega^c)$ tending to the unit disk $B_1=\{z\in\R^2:|z|<1\}$.  Then $\tilde{\mathbf{w}}_k$ have the bounded Dirichlet integrals $\lambda^{-2}D_k$, so $\tilde{\mathbf{w}}_k$ converge weakly (up to subsequence) to some function~$\mathbf{u}_E\in W^{1,2}(B_1)$.
It is easy to see that $\mathbf{u}_E$ is a solution to the Euler system in the unit disk:
\begin{equation}
\label{Eul-intr}\left\{\begin{array}{rcl} \big(\mathbf{u}_E\cdot\nabla\big)\mathbf{u}_E+\nabla p_E & = & 0 \qquad \ \ \
\hbox{\rm in }\;\;B_1,\\[4pt]
\div\mathbf{u}_E & = & 0\qquad \ \ \ \hbox{\rm in }\;\;B_1,
\\[4pt]
\mathbf{u}_E &  = & \ee_1\ \
 \qquad\  \hbox{\rm on }\;\;S_1=\partial B_1.
\end{array}\right.
\end{equation}
Denote $\e^2=\int\limits_{B_1}|\nabla\mathbf{u}_E|^2$. Then from the first equation~(\ref{Eul-intr}${}_1$), one may guess that $|p_E|\sim\e$. 
Nevertheless, surprisingly much better estimate holds
\begin{equation}
\label{de-e-intr}\sup\limits_{z_1,z_2\in \bar B_1}|p_E(z_1)-p_E(z_2)|\le C\,\e^2,
\end{equation}
for some universal constant~$C$; this estimate plays the crucial role in 
adapting the methods of paper~\cite{KPR20} to the present case.

The rest of the paper is organized as follows. In Section~\ref{sec:2}, we introduce some frequently used notations and lemmas. In Section~\ref{sec:3}, the important estimate~(\ref{de-e-intr}) for solutions to the Euler system~(\ref{Eul-intr}) is established. In Section~\ref{sec:4}, a crucial smallness of Dirichlet energy~(\ref{Dlambda--iii}) is proved which has various interesting consequences. In Section~\ref{sec:5}, we describe the blow-down argument (i.e., the convergence of the rescaled solutions~$\tilde{\mathbf{w}}_k$ to the Euler solutions, which leads to key estimates for the Bernoulli pressure on certain ``good'' circles). Finally, the proof of main Theorem~\ref{thm_main} is finished in Sections~\ref{sec:6}--\ref{sec:7} for the symmetric and for general cases respectively.

\section{Notations and preliminaries}
\label{sec:2}

\subsection{Notations}
Throughout the paper, the Reynolds number $\lambda$ will be a small positive number. The exact smallness requirement on $\lambda$ will be determined by the proof.  

We use the notation $\epsilon_k$ to represent any sequence of positive numbers that converges to $0$ as $k \to \infty$. The exact meaning of $\epsilon_k$ may change from line to line. The expression $|A_k| \le \epsilon_k$ is equivalent to $|A_k| \to 0$ as $k \to \infty$. We use $C$ to denote positive constants that are independent of $\lambda$ and $k$.

The pair $(\mathbf{w}_k, p_k)$ will always be a sequence of solutions to \eqref{NSE} on invading domains $\Omega_k$. $(\mathbf{w}_L, p_L)$ will always stand for Leray's $D$-solution in $\Omega$, that is, the weak limit of $(\mathbf{w}_k, p_k)$ as $k \to \infty$.

We use standard notations for Sobolev spaces $W^{m,q}(\Omega)$, where $m \in \mathbb{N}$, $q\in [1,+\infty]$. In our proof we do not distinguish the function spaces for scalar or vector valued functions, since it will be clear from the context which one we mean. 

\subsection{Preliminary facts for $\mathbf{w}_k$ and general $D$-solutions}

Below we list some basic estimates of $\mathbf{w}_k$ and $p_k$ known in the literature. Let $\omega_k = \partial_2 w_{k,1}- \partial_1 w_{k,2}$ be the vorticity. We remark that, as $\delta \to 0$, the constants $C_{\delta, \lambda}$ below obtained using general methods would increase to infinity. Thus, precise control of the behaviours of $\mathbf{w}_k$ and $p_k$ is not available near the outer boundary $S_{R_k}$, essentially due to the huge size of the invading domains. 

\begin{lemma} \label{lem:basic_regularity}
For any $0<\delta<\frac12$, there exists a constant $C_{\delta, \lambda}$(depending only on $\partial \Omega, \delta$ and $\lambda$) such that the following estimates hold uniformly in $k$:
\begin{enumerate}
\item  $\max_{\Omega_k \cap B_{(1-\delta) R_k}}|\mathbf{w}_k| + |p_k| \le C_{\delta, \lambda}$,
\item  $\int_{\Omega_k \cap B_{(1-\delta) R_k}} r|\nabla \omega_k|^2 \le C_{\delta, \lambda}$, 
\item  $\int_{\Omega_k \cap B_{(1-\delta) R_k}} |\nabla p_k|^2 \le C_{\delta, \lambda}$.
\end{enumerate}
\end{lemma}
\begin{proof}
For Claim \emph{1}, see \cite[Lemma 2.4, Lemma 2.6]{GW1}. For Claim \emph{2}, see \cite[Lemma 3.2]{GW1}. (The proofs in \cite{GW1} were carried out for domains like $\Omega_k \cap B_{\frac12R_k}$, but they clearly work for  $\Omega_k \cap B_{(1-\delta) R_k}$ with any $\delta > 0$ as well.) By the equation (\ref{NSE})$_1$, Claim \emph{3} is a simple corollary of \emph{1} and \emph{2}. 
\end{proof}

Next, we present two important lemmas for general $D$-solutions to the Navier-Stokes equations. They have been very useful in many previous studies on the Navier-Stokes exterior problem, see, for instance, \cite{GW, Amick, KPR, KPR20, KR21}. Lemma \ref{lem:pressure} was proved in \cite[Lemma 4.1]{GW}, and Lemma \ref{lem:angle} was proved in \cite[Theorem~4, page~399]{GW}.  Let $\mathbf{w}$ be a $D$-solution to the Navier-Stokes equations in some ring $\Omega_{r_1, r_2} = \{z \in \mathbb{R}^2: 0 < r_1 <|z| < r_2\}$, and $p$ be the corresponding pressure.

\begin{lemma}[\cite{GW}] \label{lem:pressure}
 Denote by $\bar{p}(r)$ the average of $p$ over the circle $S_r$. Then for any $r_1 < \rho_1 \le \rho_2 < r_2$, we have
 \begin{equation} \label{eq:pdifference}
  |\bar{p}(\rho_2) - \bar{p}(\rho_1)| \le \frac{1}{4\pi}  \int_{\Omega_{\rho_1, \rho_2}} |\nabla \mathbf{w}|^2
 \end{equation}
\end{lemma}

Denote by $\bar{\mathbf{w}}(r)$ the average of $\mathbf{w}$ over the circle $S_r$. Note, that the estimates for $\bar{\mathbf{w}}(r)$ are not so good: 
\begin{equation} \label{bardiff}
 |\bar{\mathbf{w}}(\rho_2) - \bar{\mathbf{w}}(\rho_1)| \le  \frac{1}{\sqrt{2\pi}}\left(\int_{\Omega_{\rho_1, \rho_2}} |\nabla \mathbf{w}|^2\right)^\frac12\cdot\left(\ln \frac{r_2}{r_1}\right)^\frac12
\end{equation}
(see, e.g., the section~2 in~\cite{KPR}). Nevertheless, the {\it direction} of $\bar{\mathbf{w}}(r)$ is still under control of the Dirichlet integral: 

\begin{lemma}[\cite{GW}] \label{lem:angle}
  Denote by $\varphi(r)$ the direction of the vector $\bar{\mathbf{w}}(r) = (\bar{w}_1(r), \bar{w}_2(r))$, i.e., 
$\bar{\mathbf{w}}(r)=|\bar\we(r)|\,(\cos\varphi(r),\sin\varphi(r))$.  Assume also that $|\bar{\mathbf{w}}(r)| \ge \sigma > 0$ for some constant $\sigma$ and for all $r \in (r_1, r_2)$. Then for any $r_1 < \rho_1 \le \rho_2 < r_2$, we have
 \begin{equation} \label{eq:phidifference}
  |\varphi(\rho_2) - \varphi(\rho_1)| \le \frac{1}{4\pi \sigma^2} \int_{\Omega_{\rho_1, \rho_2}} \left( \frac{1}{r}|\nabla \omega| + |\nabla \mathbf{w}|^2 \right) \ .
 \end{equation}
 Here $\omega = \partial_2 w_1 - \partial_1 w_2$ is the vorticity of $\mathbf{w}$.
\end{lemma}

The convergence of $\mathbf{w}_k$ to $\mathbf{w}_L$ and the basic properties of $\mathbf{w}_L$ are summarized in the following lemma. Note that the local uniform convergence is a standard fact, and the properties of $\mathbf{w}_L$ has been established through the papers \cite{GW, Amick, KPR2, KPR}.

\begin{lemma} \label{lem:wkwL}
$(\mathbf{w}_k, p_k)$ converge to $(\mathbf{w}_L, p_L)$ in $C^k$ norm on every bounded set for any $k \ge 0$. The $D$-solution $\mathbf{w}_L$ in exterior domain $\Omega$ is bounded, and both $\mathbf{w}_L$ and $p_L$ converge at infinity to some finite limits.
\end{lemma}

Denote by $\mathbf{w}_{0}$ the limit of the velocity~$\mathbf{w}_L$. Without loss of generality, below we always assume that the limit of the pressure is zero:
\begin{equation}\label{ler-limp}
\lim\limits_{|z|\to+\infty}\pl(z)=0.
\end{equation}
Then an immediate consequence of Lemma \ref{lem:wkwL} is:
\begin{corollary} \label{lem:Rk1}
There exist a sequence of radii $R_{k1}\to+\infty $ \ such that \ $\frac{R_k}{R_{k1}}\to+\infty$ and on the circle $S_{R_{k1}}$ we have
\begin{equation}\label{eq:wk-Rk1}
|\mathbf{w}_k(z) - \mathbf{w}_{0}| \le \epsilon_k\qquad\forall z\in S_{R_{k1}}, 
\end{equation}
and
\begin{equation}\label{eq:pk-Rk1}
|p_k(z) | \le \epsilon_k\qquad\forall z\in S_{R_{k1}}.
\end{equation}
\end{corollary}

In order to apply Lemma~\ref{lem:angle} , we need also the following simple observation: 
for any fixed $\delta\in (0,1)$ the uniform convergence 
\begin{equation}\label{eq:angle-vort}
\left(\,\int\limits_{\Omega_{R_{k1},(1-\delta)R_k}}\frac{1}{r}|\nabla \omega|\right)\to0\qquad \mbox{ as }k\to \infty
\end{equation}
holds.  This asymptotic estimate follows immediately from Lemma~\ref{lem:basic_regularity}${}_2$ and from the classical H\"older inequality. 

\section{On solutions to Euler equations}
\label{sec:3}

In this section we consider some properties of weak solutions $\ve\in W^{1,2}(B_1)$ to the Euler system
\begin{equation}
\label{Eul}\left\{\begin{array}{rcl} \big({\bf
v}\cdot\nabla\big){\bf v}+\nabla p & = & 0 \qquad \ \ \
\hbox{\rm in }\;\;B_1,\\[4pt]
\div{\bf v} & = & 0\qquad \ \ \ \hbox{\rm in }\;\;B_1,
\\[4pt]
{\bf v} &  = & \ee_1\ \
 \qquad\  \hbox{\rm on }\;\;S_1=\partial B_1,
\end{array}\right.
\end{equation}
where, recall, $B_1=\{z\in\R^2:|z|<1\}$ is the unit disk and $\ee_1=(1,0)$ is the unit vector. Suppose that 
\begin{equation}
\label{de}\int\limits_{B_1}|\nabla\ve|^2<\e^2
\end{equation}
for some constant $\e\in(0,1)$. Then from the first equation~(\ref{Eul}${}_1$) one can assume that $|p|\sim\e$. 
Nevertheless, surprisingly much better estimate holds (which plays the crucial role in the subsequent analysis). 

\begin{theorem} \label{th:euler_est}
{\sl Let $\ve\in W^{1,2}(B_1)$ satisfy the estimate~(\ref{de}), and let 
$(\ve,p)$ satisfy the Euler system~(\ref{Eul}${}_1$) for almost all $z\in B_1$. Then $p\in W^{2,1}(B_1) \subset C(\bar{B}_1)$, moreover,
\begin{equation}
\label{de-e}\sup\limits_{z_1,z_2\in \bar B_1}|p(z_1)-p(z_2)|\le C\,\e^2,
\end{equation}
where $C$ is some universal constant (does not depend on $\e,\ve,p$\,). }
\end{theorem}

The further proof splits into three steps, which will be written below as some separate lemmas. 

\begin{lemma} \label{lem:boundary-Eul}
\label{reg-2} {\sl Under assumptions of Theorem~\ref{th:euler_est} there exists a measurable set $\HI\subset(\frac12,1)$
such that
\begin{enumerate}
\item[(i)] \ density of $\HI$ at one equals~1:
\begin{equation}\label{md1} \frac{\meas\bigl(\HI\cap[1-h,1]\,\bigr)}{h}\to 1\qquad\mbox{ as \ }h\to0+;
\end{equation}
\item[(ii)] \qquad $\lim\limits_{\HI\ni r\to1-}\int\limits_{S_r}|\ve(z)-\ee_1|\cdot|\nabla\ve|\,ds=0,$
\item[(iii)] \qquad $\lim\limits_{\HI\ni r\to1-}\sup\limits_{z\in S_r}\bigl|\ve(z)-\ee_1\bigr|=0.$
\end{enumerate}
}
\end{lemma}

\begin{proof} By the one-dimensional Hardy's inequality, we have
\begin{equation} \label{eq:hardyineq}
\int_{B_1 \setminus B_\frac12} \frac{|{\ve} - \ee_1|^2}{(1-r)^2}  \le C \int_{B_1} |\partial_r {\ve}|^2 < +\infty,
\end{equation} 
where $r = \sqrt{x^2 + y^2}$. Therefore, for $h\in(0,1)$
\begin{equation} \label{eq:hardyineq1}
\int_{B_1 \setminus B_{1-h}}|{\ve} - \ee_1|^2 = o(h^2),
\end{equation}
and  by H\"{o}lder inequality,
$$\int_{B_1 \setminus B_{1-h}}|{\ve} - \ee_1|\cdot|\nabla\ve|=o(h).$$
The last estimate implies easily that there exists a measurable set $\HI\subset(\frac12,1)$ such that
the density of $\HI$ at one equals~1 (e.g., (\ref{md1}) holds\,), the restriction $\ve|_{S_r}$ is an absolutely continuous function of one variable for all~$r\in\HI$,  and
 \begin{equation}\label{md-a1}\lim\limits_{\HI\ni r\to1-}\int\limits_{{S}_r}|{\ve} - \ee_1|\cdot|\nabla\ve|\,ds=0.
\end{equation}
In particular,  \begin{equation}\label{md-a1-121}\lim\limits_{\HI\ni r\to1-}\int\limits_{{S}_r}\biggl|\nabla|{\ve} - \ee_1|^2\biggr|\,ds=0.
\end{equation}
From the last assertion, by virtue of the boundary condition~ (\ref{Eul}${}_3$) and from the continuity of the trace operator, we obtain immediately that 
 \begin{equation}\label{md-a2}\lim\limits_{\HI\ni r\to1-}\max\limits_{z\in{S}_r}|{\ve}(z) - \ee_1|=0.
\end{equation}
The Lemma is proved.\end{proof}

Below  in this section we assume that all the assumptions of Theorem~\ref{th:euler_est} are fulfilled. 
Taking divergence on the first equation in \eqref{Eul} gives
\begin{equation}
\Delta p = - \nabla {\ve} \cdot (\nabla {\ve})^\intercal
\end{equation}
We can extend ${\ve}$ outside $B_1$ by the constant vector $\ee_1$ so that it is globally defined and divergence free in $\mathbb{R}^2$. By the classical div-curl lemma (see, e.g., \cite{CLMS}), \,$\nabla {\ve} \cdot (\nabla {\ve})^\intercal$ belongs to the Hardy space $\mathcal{H}^1(\mathbb{R}^2)$. Moreover, $\nabla {\ve} \cdot (\nabla {\ve})^\intercal$ vanishes outside $B_1$, so that it belongs to the space $\mathcal{H}^1_z(B_1)$. Here $\mathcal{H}^1_z(B_1)$ is defined, for example, in \cite[Page 287]{CKS}, to be the space of functions on $B_1$ whose extension by 0 to whole space lies in $\mathcal{H}^1(\mathbb{R}^2)$. Using the classical theory of singular integrals and standard theory of harmonic functions, one can show that $p \in W^{2,1}_{loc}(B_1)$, see, for example, \cite[Lemma 2.11]{KPR13}. By the Euler equation and the $W^{1,2}$-regularity of ${\ve}$, we also know that $p \in W^{1,s}(B_1), s<2$.  

\begin{lemma} \label{lem:neumann}
The pressure $p$ satisfies the homogeneous Neumann boundary condition on $S_1$ in the weak sense, that is, for any $\varphi \in C^\infty(\bar{B}_1)$, we have
\begin{equation} \label{eq:neumann}
\int_{B_1} \nabla p \cdot \nabla \varphi = - \int_{B_1} (\Delta p) \varphi
\end{equation}
\end{lemma}
\begin{proof}
If ${\ve}, p$ are smooth functions in $\bar{B}_1$, then \eqref{eq:neumann} can be directly checked. Indeed, on $S_1$, there holds pointwisely
\begin{align*}
\partial_n p &= [({\ve} \cdot \nabla) {\ve}] \cdot {\mathbf n}  = (\partial_1 {\ve}) \cdot {\mathbf n} \\
& = (-\partial_2 v_2, \partial_1 v_2) \cdot {\mathbf n} \\
& = \partial_s v_2 = 0,
\end{align*}
where in the second line we have used the divergence free condition on ${\ve}$  \ (here $\partial_s$ means the tangent derivative along the circle: \  $\mathbf{s} = \mathbf{n}^\perp$).

Since we are dealing with Sobolev functions, a simple approximation argument is necessary. Now fix $\varphi \in C^\infty(\bar{B}_1)$ and consider the approximate boundary term
\begin{align}
\label{ex:tside}\int_{S_{1-\delta}} \varphi \partial_n p \ ds &= \int_{S_{1-\delta}} \varphi [({\ve} \cdot \nabla) {\ve}] \cdot\mathbf{n} \ ds \\
& =  \int_{S_{1-\delta}} \varphi [(\ee_1 \cdot \nabla) {\ve}] \cdot\mathbf{n} \ ds +   \int_{S_{1-\delta}} 
\varphi [([{\ve}-\ee_1] \cdot \nabla) {\ve}] \cdot\mathbf{n} \ ds.
\end{align}
For almost all $0<\delta<\frac12$, the gradient $\nabla\ve$ is well defined on the circle~$S_{1-\delta}$. It suffices to prove that for some sequence of $\delta \to 0$, both integrals on the right hand side converge to 0. For the first integral, we have 
\begin{align*}
\int_{S_{1-\delta}} \varphi  [(\ee_1 \cdot \nabla) {\ve}] \cdot\mathbf{n} \ ds = \int_{S_{1-\delta}} \varphi \partial_s v_2\,ds= - \int_{S_{1-\delta}} v_2\partial_s \varphi\,ds\to 0,
\end{align*}
as $\delta \to 0$, since $v_2 \equiv 0$ on $S_1$. The second integral goes to zero when $\HI\ni(1-\delta)\to1-$
because of Lemma~\ref{lem:boundary-Eul}~(iii). Thus it is easy to show that there exists a sequence of $\delta \to 0$ such that the above right hand side of~(\ref{ex:tside}) goes to 0.
\end{proof}

\begin{lemma} \label{lem:pEEEE} Under assumptions of Theorem~\ref{th:euler_est} the inclusion
$p \in W^{2,1}(B_1) \subset C(\bar{B}_1)$ holds, moreover, 
\begin{equation} \label{eq:prest2}
 \|p - \dashint_{B_1} p\|_{C(\bar{B}_1)} \le C\e^2.
 \end{equation}
\end{lemma}
\begin{proof}
By Lemma \ref{lem:neumann} and the discussion before it, $p$ is the  solution (unique in $W^{1,s}(B_1)$) to the problem
\begin{equation*}
\begin{cases}
\Delta p = -  \nabla {\ve} \cdot (\nabla {\ve})^\intercal \in \mathcal{H}^1_z(B_1) \mathrm{\ \ in \ } B_1, \\
\partial_n p = 0 \mathrm{\ \ on \ } S_1.
\end{cases}
\end{equation*} 
Let $G_N$ be the Green's function in $B_1$ with Neumann boundary conditions, then we have 
$$p(z) = \int G_N(z,z') \Delta p(z')  \ dz' +C.$$ 
Since $|\partial _z G_N(z,z')|\le \frac{C}{|z-z'|}$, by the classical   Young inequality for convolutions (see, e.g., \cite[Section~4.3]{Lieb}) we obtain 
$$\|\nabla p\|_{L^q(B_1)}\le C_q\,\e^2$$ 
for any $q\in [1,2)$. In particular, 
\begin{equation} \label{eq:pr-gr}
\|\nabla p\|_{L^1(B_1)}\le C\e^2.
\end{equation} 
Further, by the result of \cite[Theorem 6.2]{CKS}, we have
$$\|\nabla^2 p\|_{L^1(B_1)} \le C \|\Delta p\|_{\mathcal{H}^1_z(B_1)} \le C \|\nabla {\ve}\|_{L^2(B_1)}^2$$
The last two estimates together with the Sobolev embedding $W^{2,1}(B_1) \subset C(\bar{B}_1)$ imply the required estimate~(\ref{eq:prest2}) immediately (see, e.g., \cite[Lemma~2.1]{BKK}). 
\end{proof}

The proof of the central Theorem~\ref{th:euler_est} is finished. Applying in addition Lemma~\ref{lem:boundary-Eul},
 we conclude, in particular, that 

\begin{corollary} \label{lem:pE} 
\label{reg-2} {\sl  Let the assumptions of Theorem~\ref{th:euler_est} be fulfilled. Then there exists a set  $\HI_0\subset(\frac12,1)$ of positive measure such that 
\begin{equation} \label{eq:pr-gr7}
\max_{z\in S_r}\biggl(|\ve(z)-\ee_1| + \bigl|p(z) - \dashint_{S_{r}} p\bigr|\biggr) \le C \e^2\qquad\qquad\forall r\in\HI_0,
\end{equation}where $C$ is some universal constant (does not depend on $\e,\ve,p$\,). }
\end{corollary}

\section{The total and tail Dirichlet integral}
\label{sec:4}

\begin{lemma} \label{lem:Dkuniform}
For $\lambda$ sufficiently small (the smallness depends only on $\partial \Omega$), and for $k$ sufficiently large, the total Dirichlet integral of $\mathbf{w}_k$ satisfies the following uniform bound
\begin{equation}\label{eqq:Dkuniform}
D_k = \int\limits_{\Omega_k} |\nabla \mathbf{w}_k|^2 \le \frac{C \lambda^2}{|\ln \lambda|}
\end{equation}
\end{lemma}
\begin{proof}
This lemma is a version of \cite[Lemma 9]{KR21} for Navier-Stokes solutions on bounded domains. The same method in \cite{KR21} applies here as well.

To elaborate a bit, we define the following special solenoidal extension. Let $\tau \in C^\infty(\mathbb{R})$ satisfying $\tau(r) = 0$ for $r\le \frac12$ and $\tau(r) = 1$ for $r\ge 1$. Define $\mu(r) = \tau(\log r/ \log R)$, so that $\mu(r) = 0$ when $r\le \sqrt{R}$.  Define a solenoidal vector field $\mathbf{A} = (A_1, A_2)$ by
 \begin{equation*}
  A_1 = \partial_y (\lambda y \mu(|z|)), \quad A_2 = -\partial_x (\lambda y \mu(|z|)).
 \end{equation*}
Note that $\mathbf{A}$ extends the boundary data in \eqref{NSE_k}. It remains to perform an energy estimate using $\mathbf{A}$, following the arguments in \cite{KR21} line by line. Note that, we still set the parameter $R$ to be $\lambda^{-\frac14}$, just as we did in \cite{KR21}, and respectively $k$ must be sufficiently large so that $R_k>R$.
\end{proof}

For the proofs to follow, it is crucial to introduce the notion of \emph{tail} Dirichlet integral. We define the tail integral to be
\begin{equation} \label{eq:D*}
D_* := \lim_{r \to \infty} \overline{\lim_{k \to \infty}}  \int_{\Omega_k \cap \{|z| \ge  r\}} |\nabla \mathbf{w}_k|^2 .
\end{equation}
This is the portion of Dirichlet integral that ``leaks" to infinity as $k \to \infty$. Of course, $D_* \le \frac{C \lambda^2}{|\ln \lambda|}$ by Lemma \ref{lem:Dkuniform}. Let $\mathbf{F}_k, \mathbf{F}_L$ be the \emph{forces} (imposed by the fluid on the obstacle) associated with $\mathbf{w}_k$ and $\mathbf{w}_L$ respectively, that is,
$$\mathbf{F}_k =- \int_{\partial \Omega} (\nabla \mathbf{w}_k + (\nabla \mathbf{w}_k)^\intercal - p_k \mathbf{I}) \cdot \mathbf{n} \ dS, $$
and
$$\mathbf{F}_L = -\int_{\partial \Omega} (\nabla \mathbf{w}_L + (\nabla \mathbf{w}_L)^\intercal - p_L \mathbf{I}) \cdot \mathbf{n} \ dS. $$
Here, $I$ is the identity matrix, and $\mathbf{n}$ is the outward unit normal vector of $\partial \Omega$. By Lemma \ref{lem:wkwL}, we have $\mathbf{F}_k \to \mathbf{F}_L$, as $k \to \infty$. Note that in the definition of force, $\partial \Omega$ can be replaced by any closed smooth Jordan curve which encircles $\partial \Omega$. The forces are directly related to the Dirichlet integrals according to the following lemma:

\begin{lemma} \label{lem:force}
Let $\mathbf{w}_{L, \infty}$ be the limiting velocity of $\mathbf{w}_L$ at spatial infinity. Then the following statements are valid:
\begin{enumerate}
\item \begin{equation} D_k = \int_{\Omega_k} |\nabla \mathbf{w}_k|^2 = \mathbf{F}_k \cdot \mathbf{w}_\infty \to \mathbf{F}_L \cdot \mathbf{w}_\infty. \end{equation}
\item If $\mathbf{w}_0 \neq 0$, then \begin{equation} D_L := \int_{\Omega} |\nabla \mathbf{w}_L|^2 =  \mathbf{F}_L \cdot \mathbf{w}_{0}.\end{equation}
\item  If $\mathbf{w}_0 \neq 0$, then \begin{equation} D_* =  \mathbf{F}_L \cdot (\mathbf{w}_\infty - \mathbf{w}_{0}). \end{equation}
\end{enumerate}
\end{lemma}
\begin{proof}
Claim 1 follows from standard energy estimate for stationary Navier-Stokes equations on bounded domains with constant boundary values. By the result of Sazonov~\cite{S} \,(see also~\cite{GaldiS}), for the $D$-solution $\mathbf{w}_L$ in the exterior domain $\Omega$, if its limiting velocity $\mathbf{w}_{0}$ is nonzero, then $\mathbf{w}_L$ is physically reasonable in the sense of Smith~\cite[Page 350]{Sm}. In this case, the behaviour of $\mathbf{w}_L$ at infinity is controlled by that of the fundamental Oseen tensor, see~\cite[Theorem 5]{Sm}. As a consequence, the energy equality holds true, see, e.g., the proof of~\cite[Theorem 10]{Sm}. Claim 3 follows from Claims 1 and 2, together with the simple fact that
\begin{align}
D_* &= \lim_{r \to \infty}  \overline{\lim_{k \to \infty}}  \left\{ D_k -\int_{\Omega_k \cap \{|z| \le  r\}} |\nabla \mathbf{w}_k|^2  \right\} \\
&=  \overline{\lim_{k \to \infty}}  D_k - D_L.
\end{align}
\end{proof}

\section{Blow-down analysis and good circles near the outer boundary}
\label{sec:5}

The discussion before Lemma \ref{lem:basic_regularity} shows that good estimates near the outer boundary $S_{R_k}$ for $\mathbf{w}_k$ and $p_k$ are not for free. For our purpose, a blow-down argument from Navier-Stokes to Euler turns out to be crucial. The limiting Euler solution allows us to prove key estimates up to the boundary, with the help of an implicit Neumann boundary condition for the pressure, introduced in Section~\ref{sec:3}. This observation allows us to show the existence of certain \emph{good circles} (as specified later in Corollary \ref{lem:Rk2}) for the Navier-Stokes solutions near $S_{R_k}$. 

Let $\tilde{\mathbf{w}}_k(z) = \lambda^{-1} \mathbf{w}_k(R_k z)$, and $\tilde{p}_k = \lambda^{-2} p_k(R_k z)$ be functions defined in the rescaled domain $R_k^{-1} \Omega_k = B_1 \setminus R_k^{-1} \Omega^c$. As $k \to \infty$, the domain $R_k^{-1} \Omega_k$ approaches the unit disc $B_1$ as its inner boundary $R_k^{-1} \partial \Omega$ shrinks to a point $0$. Let us extend $\tilde{\mathbf{w}}_k$ by $0$ inside $R_k^{-1} \partial \Omega$ so that it is defined in $B_1$. According to Lemma \ref{lem:Dkuniform}, we have 
\begin{equation} 
\int_{B_1} |\nabla \tilde{\mathbf{w}}_k|^2  \le \frac{C}{|\ln \lambda|}
\end{equation}
Lemma \ref{lem:basic_regularity} implies that, for any $\delta > 0$,
\begin{equation}
\int_{R_k^{-1} \Omega_k \cap B_{1-\delta}} |\nabla \tilde{p}_k|^2 \le \lambda^{-4} C_{\delta, \lambda}.
\end{equation}
Hence, up to a subsequence of $k$, we have
\begin{equation}\label{eq-w}
\tilde{\mathbf{w}}_k \rightharpoonup \mathbf{u}_E, \mathrm{\ weakly \ in \ } W^{1,2}(B_1)
\end{equation}
and
\begin{equation}\label{eq-ww}
\tilde{p}_k \rightharpoonup p_E, \mathrm{\ weakly \ in \ } W^{1,2}(B_{1-\delta} \setminus B_\delta), \mathrm{\ for \ any\ } \delta > 0,
\end{equation}
for some limiting functions $\mathbf{u}_E \in W^{1,2}(B_1)$ and $p_E \in W^{1,2}_{loc}(B_1)$. By the definition of $D_*$ in \eqref{eq:D*}, we can deduce
\begin{equation} \label{eq:4a}
\int_{B_1}|\nabla \mathbf{u}_E|^2  \le \lambda^{-2} D_*
\end{equation}
Standard theory shows that $\mathbf{u}_E$ and $p_E$ solve the Euler system in $B_1$, that is,
\begin{equation} \label{eq:Euler}
\begin{cases}
\mathbf{u}_E \cdot \nabla \mathbf{u}_E + \nabla p_E = 0, \\
\nabla \cdot \mathbf{u}_E = 0.
\end{cases}
\end{equation}
Moreover, by the trace theorem and the compactness of the trace operator $W^{1,2}(B_1)\hookrightarrow L^2(S_1)$,  up to a further subsequence there holds
\begin{equation}
\tilde{\mathbf{w}}_k|_{S^1} \to \mathbf{u}_E|_{S^1}, \mathrm{\ in\ } L^2(S_1).
\end{equation}
This implies 
\begin{equation}
\mathbf{u}_E = \ee_1 \ \mathrm{\ on \ } S_1.
\end{equation}

The weak convergence (\ref{eq-w})--(\ref{eq-ww}) implies the uniform convergence on almost all circles:

\begin{lemma} \label{lem:Rk2-before} {\sl Under above notations, there exists
a~subsequence~${k_l}$ such that  $\tilde{\mathbf{w}}_{k_l}|_{S_r}\rightrightarrows \mathbf{u}_E$ and $\tilde{p}_{k_l}|_{S_r} \rightrightarrows p_E$ uniformly on
circles $S_r$ for almost all $r\in(0,1)$. }
\end{lemma}

\begin{proof}
The validity of
this assertion  was proved in~\cite[Theorem 3.2]{Amick84} (see also \cite[Lemma 3.3]{KPR13}). 
\end{proof} 

Below we assume (without loss of generality) that the subsequence
$\mathbf{w}_{k_l}$ coincides with the whole sequence $\mathbf{w}_{k}$.

A simple (but important!) consequence of Corollary~\ref{lem:pE} and Lemma~\ref{lem:Rk2-before} is:

\begin{corollary} \label{lem:Rk2}
There exists a number $0<\delta_0<\frac12$ such that on the circle $S_{1-\delta_0}$ we have
\begin{enumerate}
\item $\tilde{\mathbf{w}}_k(z) = \lambda^{-1} \mathbf{w}_k(R_k z) \rightrightarrows \mathbf{u}_E$ and $\tilde{p}_k = \lambda^{-2} p_k(R_k z) \rightrightarrows p_E$ uniformly,
\item $|\mathbf{u}_E - \ee_1| + |p_E - \dashint_{S_{1-\delta_0}} p_E| \le C\lambda^{-2} D_*$.
\end{enumerate}\end{corollary}

Now let $R_{k\delta} = (1-\delta_0) R_k$. 
As a consequence,
\begin{equation}\label{eq:f-Eul}
\lambda\cdot\bigl|\we_k(z)-\lambda\ee_1 \bigr| + \biggl|p_k(z) - \dashint_{S_{R_{k\delta}}} p_k\biggr| \le C D_*+\e_k\qquad\ \forall z\in S_{R_{k\delta}}
\end{equation}
for $k$ sufficiently large.  Moreover, since $$\biggl|\dashint_{S_{R_{k1}}} p_k-\dashint_{S_{R_{k\delta}}} p_k\biggr|\le\frac1{4\pi}D_*+\e_k\qquad\mbox{ and }\qquad\biggl|\dashint_{S_{R_{k1}}} p_k\biggr|\le\e_k$$ for large $k$ (see Lemma~\ref{lem:pressure} and Corollary~\ref{lem:Rk1}), we can rewrite the last estimate~(\ref{eq:f-Eul}) in the following more precise way: 
\begin{equation}\label{eq:fin-Eul}
\lambda\cdot\bigl|\we_k(z)-\lambda\ee_1 \bigr| + \bigl|p_k(z)\bigr| \le C D_*+\e_k\qquad\ \forall z\in S_{R_{k\delta}}.
\end{equation}
Denote $\lambda_0=|\we_0|$. Further in our arguments the Bernoulli pressure $\Phi_k=p_k+\frac12|\we_k|^2$ plays the key role. 
From estimates (\ref{eq:wk-Rk1}) and  (\ref{eq:fin-Eul}) we obtain
\begin{equation}\label{eq:fin-Eul-Bpk1}
\biggl|\Phi_k(z)-\frac12\lambda_0^2 \biggr|\le  \e_k\qquad\ \forall z\in S_{R_{k1}}.
\end{equation}
\begin{equation}\label{eq:fin-Eul-Bpkd}
\biggl|\Phi_k(z)-\frac12\lambda^2 \biggr|\le C D_*+\e_k\qquad\ \forall z\in S_{R_{k\delta}}.
\end{equation}

\begin{corollary} \label{lem:Rk3}
Under above notations  we have
\begin{equation}\label{eq:fin-Eul-vel}
\bigl|\we_k(z)\bigr| \le 2\bigl(\lambda_0+\lambda\bigr)\qquad\mbox{\rm if \ }R_{k1}\le|z|\le R_{k\delta},
\end{equation}
\begin{equation}\label{eq:fin-Eul-p}
\bigl|p_k(z)\bigr| \le C \frac{\lambda^2}{|\log \lambda|}\qquad\mbox{\rm if \ }R_{k1}\le|z|\le R_{k\delta}
\end{equation}
for $k$ sufficiently large. 
\end{corollary}

\begin{proof}[Proof of Corollary \ref{lem:Rk3}]
First of all, let us prove the estimate concerning the pressure.  By equations \eqref{NSE_k}$_1$ and \eqref{NSE_k}$_2$, the pressure solves the Poisson equation in $\Omega_k$,
 \begin{equation} \label{pPoisson}
  \Delta p_k = -\nabla \mathbf{w}_k \cdot (\nabla \mathbf{w}_k)^\intercal.
 \end{equation}
 Let $p_{1k}$ be the potential solution to \eqref{pPoisson}, i.e.,
 $$p_{1k}(z) = - \frac{1}{2 \pi} \int_{\Omega_k} \log |z-\zeta| (\nabla \mathbf{w}_k \cdot (\nabla\mathbf{w}_k)^\intercal)(\zeta) \, d\zeta_1\, d\zeta_2.$$
 By the classical div-curl lemma (see, e.g., \cite{CLMS}), \,$\nabla \mathbf{w}_k \cdot (\nabla \mathbf{w}_k)^\intercal$ belongs to the Hardy space $\mathcal{H}^1(\mathbb{R}^2)$. Hence, by the Calder\'on-Zygmund theorem for Hardy spaces~\cite{St}, $\nabla^2 p_{1k} \in L^1(\mathbb{R}^2)$, and $\nabla p_{1k} \in L^2(\mathbb{R}^2)$. Moreover, $p_{1k} \in C(\mathbb{R}^2)$ and converges to~$0$ at infinity. In particular,
 \begin{equation} \label{p1Sup}
   \sup_{\mathbb{R}^2} |p_{1k}| \le C \|\nabla \mathbf{w}_k \cdot (\nabla \mathbf{w}_k)^\intercal\|_{\mathcal{H}^1} \le C D_k.
 \end{equation}
Let $p_{2k} = p_k - p_{1k}$ be a function defined in $\Omega_k$. Clearly, $p_{2k}$ is a harmonic function and satisfies
\begin{equation}\label{eq:fin-Eul-p1}
\bigl|p_{2k}(z)\bigr| \le C D_k\qquad\ \forall z\in S_{R_{k1}}\cup  S_{R_{k\delta}}
\end{equation}
(see (\ref{eq:fin-Eul}), (\ref{eq:pk-Rk1})\,). Then by maximum principle for harmonic functions, 
\begin{equation}\label{eq:finEul-p2}
\bigl|p_{2k}(z)\bigr| \le C D_k\qquad\mbox{\rm if \ }R_{k1}\le|z|\le R_{k\delta}.
\end{equation}
The last inequality and (\ref{p1Sup}), (\ref{eqq:Dkuniform}) imply the required~(\ref{eq:fin-Eul-p}). 

Now the first estimate (\ref{eq:fin-Eul-vel}) concerning velocity follows immediately from estimates (\ref{eq:fin-Eul-Bpk1})--(\ref{eq:fin-Eul-Bpkd}) on boundary circles $S_{R_{k1}}$ and $S_{R_{k\delta}}$, and from the one-sided maximum principle for the Bernoulli pressure~$\Phi_k=p_k+\frac12|\we_k|^2$ (see, e.g., \cite{GW}\,). 
\end{proof}

Below we always assume (without loss of generality) that the estimates (\ref{eq:fin-Eul})--(\ref{eq:fin-Eul-p}) are fulfilled for all~$k$.

\section{Proof of Theorem \ref{thm_main} in the symmetric case}
\label{sec:6}

Throughout this section, we assume that $\Omega$ is symmetric with respect to the $x$-axis. Also, assume that $\mathbf{w}_k$ are symmetric with respect to the $x$-axis. This means that, as a vector field, $\mathbf{w}_k=(w^1_k,w^2_k)$ is invariant under the reflection of $\Omega$ given by $(x,y) \mapsto (x,-y)$, that is, $w^1_k(x,y)=w^1_k(x,-y)$ is an~even function of~$y$, and $w^2_k(x,y)=-w^2_k(x,-y)$ is an~odd function of~$y$. In such a symmetric situation\footnote{The usual existence theorem for Navier-Stokes equations on bounded domains holds true in the symmetric case as well, which produces symmetric solutions.}, we can give a rather simple proof for the main result.

\begin{proof}[Proof of Theorem \ref{thm_main}(symmetric case)]
We prove by contradiction. Assume that there exists a sequence of $R_k \to \infty$ such that $\mathbf{w}_k$ weakly converges to a symmetric $D$-solution $\mathbf{w}_L$ which does not equal to the Finn-Smith solution $\we_{FS}(z, \lambda)$ constructed in \cite{FS}. By the works \cite{GW, Amick, KPR} (see the discussions in the Introduction), $\mathbf{w}_L$ achieves some finite velocity at infinity. By the symmetry assumption, there exists some constant $\tilde\lambda_0\in\R$ such that
$$\lim_{z \to \infty} \mathbf{w}_L(z) = \mathbf{w}_{0} = \tilde\lambda_0 \ee_1.$$
Our assumption implies that $\mathbf{w}_0 \neq \mathbf{w}_{\infty}$, since otherwise, by the uniqueness result in \cite[Theorem 2]{KR21}, $\mathbf{w}_L$ will be identical to $\mathbf{w}_{FS}(z, \lambda)$. 

By Lemma \ref{lem:force}, we have
\begin{equation} \label{eq:5b}
\lim_{k\to \infty} D_k = \mathbf{F}_L \cdot \we_\infty 
\end{equation}
Recall that $\mathbf{F}_L$ stands for the force of $\we_L$, and is parallel to the $x$-axis by the symmetry assumption. In view of Lemma \ref{lem:Dkuniform}, this implies 
\begin{equation} \label{eq:5e}
|\mathbf{F}_L| \le \frac{C\lambda}{|\ln \lambda|}.
\end{equation}
If $\we_0 \neq 0$, then due to Lemma \ref{lem:force}  we have
\begin{equation*} \label{eq:5a}
D_* = \mathbf{F}_L \cdot (\we_\infty - \we_0)
\end{equation*} If $\we_0 = 0$, then the inequality
\begin{equation} \label{eq:5f}
D_* \le \mathbf{F}_L \cdot \we_\infty = \mathbf{F}_L \cdot (\we_\infty - \we_0).
\end{equation}
is valid. We are going to use \eqref{eq:5f} which holds true in both cases. By \eqref{eq:5f}, \eqref{eq:5b} and the evident fact $0 \le D_* \le \lim_{k\to \infty} D_k$, we have
\begin{equation} \label{eq:5g}
0 \le \tilde\lambda_0 < \lambda.
\end{equation}
In other words, $\tilde\lambda_0=\lambda_0=|\we_0|<\lambda$. So we have 
\begin{equation} \label{eq:5f-tt}
D_* \le  C\frac{\lambda}{|\ln \lambda|} (\lambda-\lambda_0).
\end{equation}
Then from the previous estimates (\ref{eq:fin-Eul-Bpk1})--(\ref{eq:fin-Eul-Bpkd}) for the~Bernoulli pressure we obtain
\begin{equation}\label{eq:case1-Rk1}
\max\limits_{z\in S_{R_{k1}}}\Phi_k(z)\le \frac12{\lambda_0}^2+\epsilon_k, 
\end{equation}
and
\begin{equation}\label{eq:case1-Rkd}
\min\limits_{z\in S_{R_{k\delta}}}\Phi_k(z)\ge \frac12{\lambda_0}^2+\frac{(\lambda-\lambda_0)}2(\lambda+\lambda_0)(1+\e_\lambda),
\end{equation}
where $\e_\lambda\sim\frac{C}{\ln\lambda}$ goes to zero uniformly as~$\lambda\to0$. 
In particular, for sufficienly small $\lambda$ we have 
\begin{equation}\label{eq:case1-Rkf}
\frac\lambda3(\lambda-\lambda_0)+\max\limits_{z\in S_{R_{k1}}}\Phi_k(z)< \min\limits_{z\in S_{R_{k\delta}}}\Phi_k(z).
\end{equation}

To proceed, we need an elegant result proved in the celebrated paper \cite[Section 4.2]{Amick} by Ch. Amick. He showed that, under the symmetry assumption, there exist two (piecewise regular) curves $M_{1k}, M_{2k} \subset \{ 1 \le |z| \le R_k\} \cap \{\omega = 0\}$, both starting at the the circle $S_1$ and ending at the large circle $S_{R_k}$. In addition, the Bernoulli pressure $\Phi_k$ is monotone increasing and decreasing along $M_{1k}$ and $M_{2k}$ respectively. Since $M_{1k}$ and $M_{2k}$ must intersect the two ``good circles" \ $S_{R_{k1}}$  \, and $S_{R_{k\delta}}$, we immediately 
obtain the desired contradiction with~(\ref{eq:case1-Rkf}).
\end{proof}

\section{Proof of Theorem \ref{thm_main} in the non-symmetric case}

\label{sec:7}
Fix $\we_\infty=\lambda\ee$ and take the limiting Leray solution $(\wl,\pl)$. Recall, that by Lemma~\ref{lem:wkwL},
there is a uniform limit
\begin{equation}\label{ler-lim}
\lim\limits_{|z|\to+\infty}\wl(z)=\ww\in\R^2,
\end{equation}
and the pressure $\pl$ has the uniform zero limit as well:
\begin{equation}\label{ler-limp1}
\lim\limits_{|z|\to+\infty}\pl(z)=0.
\end{equation}
Suppose that the assertion of Theorem \ref{thm_main} is not true, i.e., $\ww\ne\wi$. Under notations of previous sections~\ref{sec:2}--\ref{sec:5}, we have to consider two different cases. 

\subsection{Case I: there exists a sequence $R_{0k}\in\bigl[R_{k1},R_{k}\bigr]$ such that}
\begin{equation}\label{eq:case1}
\biggl|\dashint_{S_{R_{0k}}} \we_k\biggr|<\frac15\lambda.
\end{equation}
(In particular, it includes the case $|\we_0|=\lambda_0< \frac15\lambda$, here we can simply take $R_{0k}:=R_{k1}$.) Note, that (\ref{bardiff}) implies 
$R_{0k}\le R_k\cdot\exp\bigl(-4\frac{\lambda^2}{D_*+\e_k}\bigr)\ll\frac12 R_k<R_{k\delta}$ for $k$ big enough.

In this case, the desired contradiction can be obtained by repeating the corresponding arguments of paper~\cite{KPR20} almost word for word. More precisely, in the paper \cite{KPR20}, a contradiction was obtained under the assumption that the Dirichlet integrals tend to zero: $D_k\to 0$ as $k\to\infty$.  Instead of this, now we have extra-smallness of the Dirichlet integrals and extra-smallness of the pressure:
\begin{equation}\label{eq:case1-1}
D_k = \int\limits_{\Omega_k} |\nabla \mathbf{w}_k|^2 \le \frac{C \lambda^2}{|\log\lambda|}
\end{equation}
\begin{equation}\label{eq:case1-2}
\bigl|p_k(z)\bigr| \le C \frac{\lambda^2}{|\log \lambda|}\qquad\mbox{\rm if \ }R_{k1}\le|z|\le R_{k\delta}=(1-\delta_0)R_k
\end{equation}
(see above~(\ref{eqq:Dkuniform}), (\ref{eq:fin-Eul-p})\,). It is easy to check that  these smallness and near-boundary condition
\begin{equation}\label{eq:case1-3}
\bigl|\we_k(z)-\lambda\ee_1 \bigr| <\frac{C \lambda}{|\log\lambda|}\qquad\ \forall z\in S_{R_{k\delta}}
\end{equation}
(see~(\ref{eq:f-Eul})\,) are sufficient to produce the desired contradiction with~(\ref{eq:case1}) \,(for sufficiently small~$\lambda$) in exactly the same way as in~\cite{KPR20}, even with some simplifications.  (For a reader's convenience, we repeat the main ideas of~\cite{KPR20} in 
the end of the present paper.)

\subsection{Case II: the estimate from below}
\begin{equation}\label{eq:case2}
\biggl|\dashint_{S_{r}} \we_k\biggr|\ge\frac15\lambda\qquad\forall r\in\bigl[R_{k1},R_{k}\bigr]
\end{equation}
{\bf holds.} In particular, now we have 
\begin{equation}\label{eq:case2-1}
\lambda_0=|\we_0|\ge\frac15\lambda>0. 
\end{equation}
This case is much more delicate and subtle. Again we have extra-smallness 
conditions~(\ref{eq:case1-1})--(\ref{eq:case1-2}). The subsequent proof is split into a number of steps.  Below we will see, that the assumption~(\ref{eq:case2}) plays the crucial role in order to control the direction of the vector~$\we_0$ (see Step~4), and, finally, reduce this case to the situation, similar to the symmetric case (see Steps~6--7).

\

{\sc Step 1.}  {\sl We claim, that the estimate from above~
\begin{equation}\label{eq:case2-e1}
\lambda_0\le C\sqrt{\lambda}
\end{equation}
holds}. 

Indeed, under our assumptions,  the Bernoulli pressure $\Pl=\pl+\frac12|\wl|^2$ of the Leray solution goes to zero as well:
\begin{equation}\label{ler-limf}
\lim\limits_{|z|\to+\infty}\Pl(z)=0.
\end{equation}
Then the desired estimate follows immediately from the smallness of the Dirichlet integral $D_L=\int\limits_\Omega|\nabla\we_L|^2<\frac{C \lambda^2}{|\log\lambda|}$ and from Lemmas~11,13 of paper~\cite{KR21} (note, that cited Lemma~13 of \cite{KR21} is a reformulation of \cite[Theorem 11]{Amick84}). 

\medskip

Of course, the obtained estimate~(\ref{eq:case2-e1}) is far from being optimal, but it has very important implications.

\

{\sc Step 2.}  {\sl If $\lambda$ is small enough, then the Leray solution $\we_L$ coincides with the Finn--Smith solution $\we_{FS}(z, \we_0)$}.

This claim follows directly from the uniqueness result of~\cite{KR21}. Now we can obtain the formula for the force. 

\ 

{\sc Step 3.}  {\sl If $\lambda$ is small enough, then the force $\mathbf{F}_L$ satisfies the asymptotic identity}
\begin{equation}\label{eq:case2-e3}
\FF=\bigl(4\pi\ee_0+\boldsymbol{\epsilon}_\lambda\bigr)\frac{\lambda_0}{|\ln\lambda_0|}.
\end{equation}
Here $\ee_0=\frac{\we_0}{|\we_0|}=\frac1{\lambda_0}\we_0$ is the corresponding unit vector, and the vector 
$\boldsymbol{\epsilon}_\lambda$ tends to zero uniformly as $\lambda\to0$. This important identity was established in~\cite[Theorem~5.4]{FS} (see also~\cite[page 27]{FS}). 

Denote by $\varphi_0$ the angle direction of the vector $\we_0$, i.e., 
$$\we_0=\lambda_0(\cos\vp_0,\sin\vp_0)=\lambda_0\ee_0.$$

{\sc Step 4.}  {\sl The estimate
\begin{equation}\label{eq:case2-e4}
|\vp_0|\le\frac{C}{\lambda^2}D_*\le\frac{C}{|\ln\lambda|}
\end{equation}
holds. }

This important estimate follows immediately form Lemma~\ref{lem:angle} and formulas~(\ref{eq:angle-vort}), (\ref{eq:wk-Rk1}),  (\ref{eq:case2}),  (\ref{NSE_k}${}_4$).

Recall, that
\begin{equation}\label{eq:case2-e5'} D_* =  \mathbf{F}_L \cdot (\mathbf{w}_\infty - \mathbf{w}_{0})=\frac{\lambda_0}{|\ln\lambda_0|}\bigl(4\pi\ee_0+\boldsymbol{\epsilon}_\lambda\bigr)\cdot(\lambda\ee_1-\lambda_0\ee_0)\end{equation}
(see Lemma~\ref{lem:force}). 
In particular, if $D_*=0$, then from  (\ref{eq:case2-e4})--(\ref{eq:case2-e5'}) we obtain that $\vp_0=0$ and $\lambda=\lambda_0$, i.e., $\we_0=\we_\infty$, a contradiction. So below we always assume that
$$D_*>0.$$ 
We can rewrite~(\ref{eq:case2-e5'}) as
\begin{equation}\label{eq:case2-e5}
D_*=\frac{\lambda_0}{|\ln\lambda_0|}\biggl(4\pi(\lambda\cos\vp_0-\lambda_0)+\boldsymbol{\epsilon}_\lambda\cdot\bigl[(\lambda,0)-\lambda_0(\cos\vp_0,\sin\vp_0)\bigr]\biggr).
\end{equation}
Then we have, in particular, that 

\

{\sc Step 5.}  {\sl The estimate
\begin{equation}\label{eq:case2-e6}
\lambda_0\le2\lambda
\end{equation}
holds for sufficiently small $\lambda$. }

Thus we can rewrite (\ref{eq:case2-e4}) as 
\begin{equation}\label{eq:case2-e7}
|\vp_0|\le\frac{C}{\lambda^2+\lambda_0^2}D_*. 
\end{equation}

Now we a ready to obtain the most important inequality for this section (cf. with~(\ref{eq:5f-tt})\,): 

\

{\sc Step 6.}  {\sl The estimate
\begin{equation}\label{eq:case2-e8}
D_*\le 5\pi\frac{\lambda_0}{|\ln\lambda_0|}(\lambda-\lambda_0)
\end{equation}
holds for sufficiently small~$\lambda$. }

\begin{proof}
Indeed, rewrite (\ref{eq:case2-e5'}) as
\begin{equation}\label{eq:case2-e9}
D_*= \frac{\lambda_0}{|\ln\lambda_0|}(\lambda-\lambda_0)\bigl(4\pi\ee_0+\boldsymbol{\epsilon}_\lambda\bigr)\cdot\ee_0+\frac{\lambda_0}{|\ln\lambda_0|}\lambda\bigl(4\pi\ee_0+\boldsymbol{\epsilon}_\lambda\bigr)\cdot(\ee_1-\ee_0)= I+II.
\end{equation}
Since $|\ee_1-\ee_0|\le |\vp_0|\le\frac{C}{\lambda^2+\lambda_0^2}D_*$, the second term is of size $o(D_*)$ (it means, that 
$II=\e_\lambda\cdot D_*$, where the value $\e_\lambda\to0$ uniformly as $\lambda\to0$) . Therefore, 
\begin{equation}\label{eq:case2-e10}
D_*(1+\e_\lambda)= \frac{\lambda_0}{|\ln\lambda_0|}(\lambda-\lambda_0)\bigl(4\pi\ee_0+\boldsymbol{\epsilon}_\lambda\bigr)\cdot\ee_0=
\frac{\lambda_0}{|\ln\lambda_0|}(\lambda-\lambda_0)\bigl(4\pi+\e_\lambda\bigr).\end{equation}
So the desired estimate~(\ref{eq:case2-e8}) holds when $\lambda$ is small enough. 
\end{proof}
In particular, now we have 
\begin{equation}\label{eq:case2-e11}
\lambda-\lambda_0>0.
\end{equation}
Recall, that $\Phi_k=p_k+\frac12|\we_k|^2$ means the Bernoulli pressure. The next property sums up the final result of our struggle:  it is  a reward, in a sense, for all previous labors.
\

{\sc Step 7.}  {\sl The estimates 
\begin{equation}\label{eq:case2-Rk1}
\max\limits_{z\in S_{R_{k1}}}\Phi_k(z)\le \frac12{\lambda_0}^2+\epsilon_k, 
\end{equation}
and
\begin{equation}\label{eq:case2-Rkd}
\min\limits_{z\in S_{R_{k\delta}}}\Phi_k(z)\ge \frac12{\lambda_0}^2+\frac{(\lambda-\lambda_0)}2(\lambda+\lambda_0)(1+\e_\lambda).
\end{equation}
hold for sufficiently small~$\lambda$. }

These estimates follow directly from  inequality~(\ref{eq:case2-e8}) of the  Step~6 and from previous estimates  (\ref{eq:fin-Eul-Bpk1})--(\ref{eq:fin-Eul-Bpkd}). In particular, we have 
\begin{equation}\label{eq:case2-Rkf}
\frac\lambda2(\lambda-\lambda_0)+\max\limits_{z\in S_{R_{k1}}}\Phi_k(z)< \min\limits_{z\in S_{R_{k\delta}}}\Phi_k(z)
\end{equation}
for $k$ sufficiently large. This means, that we are faced with the situation, considered in the paper~\cite{KPR20}, and we can obtain the desired contradiction just by repeating the corresponding arguments of that paper. For a reader's convenience, we recall the main ideas of the proof in~\cite{KPR20} adapted for the present paper. 

From (\ref{eq:case2-Rkf}) it follows immediately, that there are two closed regular level sets $S'_{k}$ and $S''_{k}$ of $\Phi_k$ 
such that:

\begin{enumerate}
\item[(i)] \ $\Phi_k|_{S'_k}\equiv t'_k=\frac12\lambda_0^2+\e_k$, \ \ \ \ \ $\Phi_k|_{S''_k}\equiv t''_k\ge t'_k+\frac\lambda2(\lambda-\lambda_0)$. 
\item[(ii)] \qquad $S'_k$, $S''_k$ \ are smooth closed curves (homeomorphic to the circle) surrounding the origin, both of them lie between circles $S_{R_{k1}}$ and $S_{R_{k\delta}}$.  
\item[(iii)] \qquad the vorticity $\omega_k(z)$ does not change sign between the curves $S'_k$, $S''_k$; for definiteness, we can assume without loss of generality that \begin{equation}\label{eq:case2-Rkf1}
\omega_k(z)>0\qquad\mbox{ for all $z$ between  $S'_k$ and $S''_k$}
\end{equation}
\end{enumerate}
(for the last assertion, see \cite[Step~6]{KPR20}\,). 

Now denote $\dashint_{S_{r}} \we_k=\bar\we_k(r)=|\bar\we_k(r)|\,\bigl(\cos\vp_k(r),\sin\vp_k(r))$. Then by the same reasons as for inequality~(\ref{eq:case2-e4}) of the present paper, we have 
\begin{equation}\label{eq:case2-Rkf1}
|\vp_k(r)|\le\frac{C}{\lambda^2}D_*=\e_\lambda\qquad\forall r\in[R_{k1},R_{k\delta}].
\end{equation}
Hence, it is not difficult to prove, that there exists a unit vector $\tilde\ee=(\cos\tilde\theta,\sin\tilde\theta)$ such that the segment 
$L=\{s\tilde\ee:s\in[R_{k1},R_{k\delta}]\}$ satisfies the following properties: 
\begin{enumerate}
\item[(iv)] \ $\we^\bot_k(z)\cdot\tilde\ee<0$ for any $z\in L$, where we denote $(a,b)^\bot=(-b,a)$;  
\item[(v)] \ $\int\limits_L|\nabla\omega_k|\,ds\le\e_k$.  
\end{enumerate}
Take two points $A\in L\cap S'_k$ \ and \ $B\in L\cap S''_k$ such that the line segment $[A,B]$ lies between the curves $S'_k$ and $S''_k$. 
Recall, that the gradient of the Bernoulli pressure satisfies the identity
$$\nabla\Phi_k\equiv-\nabla^\bot\omega_k+\omega_k\cdot\we^\bot_k.$$
Then we have 
\begin{equation}
\label{en25}
\begin{array}{c}
t''_k-t'_k=\Phi_k(B)-\Phi_k(A)=\int\limits_{[A,B]}\nabla\Phi_k\cdot\et\,ds\\[12pt]
=-\int\limits_{[A,B]}\nabla^\bot\omega_k\cdot\et\,ds+
\int\limits_{[A,B]}\omega_k\we_k^\bot\cdot\et\,dr:=I+II.
\end{array}
\end{equation}
Estimate the~terms~$I$ and $II$ separately. From the above property (v) we have
\begin{equation}
\label{en26}
I\le \e_k.
\end{equation}
On the other hand, from (iii)--(iv) we obtain
\begin{equation}
\label{en27}
II<0.
\end{equation}
Therefore, 
$$t''_k-t'_k\le\e_k,$$
a contradiction with~(i). 

Of course, the above items (i)--(v) are only short description. In case of interest, a reader can find a detailed justification for all these steps in~\cite{KPR20}.

\

The proof of  Theorem~\ref{thm_main} is finished completely. \hfill $\qed$

\bibliographystyle{plain}

\end{document}